\documentclass[11pt]{article}
\usepackage{latexsym,amsmath}
\def\fullpage {
\addtolength{\topmargin}{-2 cm}
\addtolength{\oddsidemargin}{-1.4 cm} \addtolength{\textwidth}{+2.8 cm}
\addtolength{\textheight}{+3.3 cm}}
\fullpage
\usepackage{amsthm}
\usepackage{amssymb}
\usepackage{epsfig}
\usepackage{graphicx}

\usepackage{amssymb}
\usepackage{multicol}
\usepackage{graphicx}
\usepackage{amsmath, epsfig}
\usepackage{amsthm}
\newtheorem{theorem}{Theorem}

\newtheorem{proposition}{Proposition}
\newtheorem{lemma}{Lemma}

\newtheorem{remark}{Remark}

\newfont{\bss}{cmss12}

\newcommand{\ignore}[1]{}

\renewcommand{\phi}{\varphi}

\def\qed{{\hfill\hphantom{.}\nobreak\hfill$\Box$}}

\def\min{\mathop{\mathrm{min}}\nolimits}
\def\1{\mathrel{\mathbf{1}}}

\newcommand{\bt}{\begin{theorem}}
\newcommand{\et}{\end{theorem}}
\newcommand{\btt}{\begin{ttheorem}}
\newcommand{\ett}{\end{ttheorem}}

\newcommand{\bcc}{\begin{conjecture}}
\newcommand{\ecc}{\end{conjecture}}
\newcommand{\bc}{\begin{corollary}}
\newcommand{\bl}{\begin{lemma}}
\newcommand{\ec}{\end{corollary}}
\newcommand{\el}{\end{lemma}}
\newcommand{\bq}{\begin{question}}
\newcommand{\eq}{\end{question}}
\newcommand{\bp}{\begin{proposition}}
\newcommand{\ep}{\end{proposition}}
\newcommand{\br}{\begin{remark}}
\newcommand{\er}{\end{remark}}

\newcommand{\PG}{\ensuremath{PG}}

\newcommand{\mL}{\mathcal{L}}
\newcommand{\mP}{\mathcal{P}}

\begin{document}
 \date{}

\title{Small Maximal Partial Ovoids \\ in \\ Generalized Quadrangles}

\author{Jeroen Schillewaert \thanks{Research supported by Marie Curie IEF grant GELATI (EC grant nr 328178)}\\
\small Department of Mathematics\\ [-0.8ex]
\small Imperial College\\ [-0.8ex]
\small London, U.K.\\
\small\tt jschillewaert@gmail.com\\
\and
Jacques Verstraete \thanks{Research supported by NSF Grant DMS 1101489.}\\
\small Department of Mathematics\\[-0.8ex]
\small University of California at San Diego\\[-0.8ex]
\small California, U.S.A.\\
\small \tt jverstraete@math.ucsd.edu
}

\maketitle

\begin{abstract}
A {\em maximal partial ovoid} of a generalized quadrangle is a maximal set of points no two of which
are collinear.  The problem of determining the smallest size of a maximal partial ovoid in quadrangles has been extensively studied in the literature.
In general, theoretical lower bounds on the size of a maximal partial ovoid in a quadrangle of order $(s,t)$ are linear in $s$.
In this paper, in a wide class of quadrangles of order $(s,t)$ we give a construction of a maximal partial ovoid of size
at most $s \cdot \mbox{polylog}(s)$, which is within a polylogarithmic factor of theoretical lower bounds.
 The construction substantially improves previous quadratic upper bounds in quadrangles of order $(s,s^2)$, in particular in
 the well-studied case of the elliptic quadrics $Q^-(5,s)$. \end{abstract}

\section{Introduction}

Generalized polygons were introduced in Tits' seminal paper on triality \cite{Tits}, and encompass as special cases projective planes and generalized quadrangles. Only a small set of examples of generalized quadrangles is known: finite thick generalized quadrangles of order $(s,t)$ are known to exist only when $(s,t),(t,s) \in \{(q,q^2),(q^2,q^3),(q,q),(q-1,q+1)\}$ and $q$ is a prime power.
The existence of further generalized quadrangles is one of the main open problems in projective geometry.
A standard reference on generalized quadrangles is Payne and Thas~\cite{PT} and for generalized polygons see Van Maldeghem \cite{HVM}.

\bigskip

A {\em partial ovoid} of a generalized quadrangle is a set of points no two of which
are collinear. The problem of determining the smallest size of a maximal partial ovoid in generalized quadrangles has been very extensively studied in the literature (see De Beule, Klein, Metsch and Storme~\cite{deBKMS} and the references therein). Ovoids have a long history going back to the seminal work of Tits, and their construction and classification is the subject of many research articles over the last fifty years.

\bigskip

Particular attention has been given to maximal partial ovoids in $Q^-(5,q)$. Ebert and Hirschfeld~\cite{EH} showed that any maximal partial ovoid in $Q^-(5,q)$ has size at least $2q + 1$ if $q \geq 2$ and size at least $2q + 2$ if $q \geq 5$. Maximal partial ovoids of size $q^2 + 1$ in $Q^-(5,q)$ are easily constructed, namely elliptic quadrics $Q^-(3,q) \subset Q^-(5,q)$ consisting of points collinear with a given point of $Q^-(5,q)$,
and further examples of the same size are given by Aguglia, Cossidente and Ebert~\cite{ACE}.
The best upper bound so far for $Q^-(5,q)$ was given by Metsch and Storme  (see Theorem 3.7 of \cite{MS}) in the case $q=2^{2h+1}, h \geq 1$, and in this case a maximal partial ovoid of size roughly $\tfrac{1}{2}q^2$ was constructed.
This construction is geometric and the restriction on $q$ comes from the fact that Suzuki-Tits \cite{TitsSuzuki} ovoids are used.

 \subsection{Main Theorem}

 In a generalized quadrangle of order $(s,t)$, a simple counting argument shows that a maximal partial ovoid always has size at least $(1 + s + st + s^2t)/(1 + s + st)$.  A quadrangle of order $(s,t)$ is called {\em locally sparse} if for any set of three points, the number of points collinear with all three points is at most $s + 1$.
It is well-known that every $(s,s^2)$ quadrangle is locally sparse due to a result by Bose and Shrikhande~\cite{BS}, and in particular $Q^-(5,q)$ is locally sparse. However, it is known that in some cases, there may be sets of three points collinear with up to $t + 1$ other points -- for instance, in $H(4,q^2)$, which is a quadrangle of order $(s,t) = (q^2,q^3)$, there are triples of points all collinear with the same $q^3 + 1$ other points~\cite{PT}. In this paper, we show that locally sparse $(s,t)$ quadrangles with $t \geq s(\log s)^{2\alpha}$ have maximal partial ovoids which are within a polylogarithmic factor of the simple counting bound given above:

\medskip

\begin{theorem}\label{main}
For any $\alpha > 4$, there exists $s_{\rm o}(\alpha)$ such that if $s \geq s_{\rm o}(\alpha)$ and $t \geq s(\log s)^{2\alpha}$, then any locally sparse generalized quadrangle of order $(s,t)$ has a maximal partial ovoid of size at most
$s(\log s)^{\alpha}$.
\end{theorem}

\medskip

Since $Q^-(5,q)$ is locally sparse, this theorem improves the upper bound of roughly $\frac{1}{2}q^2$ for $Q^-(5,q)$ to $q(\log q)^{\alpha}$ for large $q$, whilst the simple counting bound as well as the bounds in~\cite{EH} are linear in $q$. In fact, as we do not require the quadrangle to be classical, we also provide the first non-trivial upper bounds in other quadrangles of order $(s,s^{2})$, in particular for Kantor's quadrangles \cite{Kantor} and the $T_{3}(O)$ of Tits \cite{TitsSuzuki}.

Due to the {\em point-line duality}\index{generalized quadrangle!point-line duality} for quadrangles in which the words ``point'' and ``line'' are interchanged, Theorem \ref{main} shows that a quadrangle of order $(s,t)$, with the property that for any three lines, at most $t + 1$ lines intersect all three of those lines and $s \geq t(\log t)^{2\alpha}$, has a maximal partial spread of size at most $t(\log t)^{\alpha}$. For instance, the quadrangle $H(3,q^2)$ dual to $Q^-(5,q)$ has a maximal partial spread of size at most $q(\log q)^{\alpha}$.

\medskip

The proof of Theorem \ref{main} gives an efficient two-step randomized algorithm for finding a maximal partial ovoid of size at most $s(\log s)^{\alpha}$ asymptotically almost surely. Probabilistic methods have been used in projective geometry by G\'{a}cs and Sz\"{o}nyi~\cite{GS} for constructing small maximal partial spreads in projective spaces, and by Kim and Vu~\cite{KV} for constructing small complete arcs in projective planes by developing
concentration inequalities for certain non-Lipschitz functions. The probabilistic approach used here is slightly less technical than in~\cite{KV}, due to the careful application of a martingale inequality for non-Lipschitz functions of independent random variables.


This paper is organized as follows: in Section \ref{prelim} we state the geometric and probabilistic preliminaries required to prove
Theorem \ref{main}, essentially concentration of measure inequalities, and in Section \ref{proof} we prove Theorem \ref{main}.

\subsection{Notation}

Throughout the paper, $Q$ will denote a thick quadrangle of order $(s,t)$, namely with $s,t \geq 2$. The set of lines of $Q$ is denoted $\mL$ and the set of points is denoted $\mP$. For a set $R \subseteq \mP$, let $R^{\perp}$ denote the set of points collinear with all points in $R$, and let $R^{\perp}_{\circ} = R^{\perp} \backslash R$. Moreover, $R^{\bowtie}$ will denote the set of points collinear with at least one point in $R$ and let $R^{\bowtie}_{\circ} = R^{\bowtie} \backslash R$. If $R = \{u\}$, then we write $u^{\perp}$ instead of $\{u\}^{\perp}$ and $u^{\perp}_{\circ}$ instead of $\{u\}^{\perp}_{\circ}$, so that
\[ R^{\perp} = \bigcap_{x \in R} x^{\perp} \quad \mbox{ and } \quad R^{\bowtie} = \bigcup_{x \in R} x^{\perp}.\]

For the material to follow, if $x$ is a real number let $f_k(x) := x(x - 1)(x - 2) \dots (x - k + 1)$ denote the {\em $k$-th falling factorial} where $f_0(x) := 1$ for all $x$. Let $\mathbb N = \{0,1,2,\dots\}$.
For functions $f,g : \mathbb N \rightarrow \mathbb R$ we write $f \sim g$ if $\lim_{n \rightarrow \infty} f(n)/g(n) = 1$ and
$f \lesssim g$ if $\limsup_{n \rightarrow \infty} f(n)/g(n) \leq 1$.

\medskip

If $X : \Omega \rightarrow \mathbb R$ is a random variable and $D \subset \mathbb R$ then $[X \in D]$ denotes the event $\{\omega \in \Omega : X(\omega) \in D\}$. For an infinite sequence of events $(A_n)_{n \in I}$, we say that $A_n$ occurs {\em asymptotically almost surely (a.a.s.)} if $\mathbb P(A_n) \rightarrow 1$ as $n \rightarrow \infty$ where the limit is taken over $n \in I$.
When we use expressions such as $X_u \lesssim g(s)$ a.a.s  we mean that for all $\epsilon > 0, X_u \leq (1 + \epsilon)g(s)$ a.a.s.

\section{Preliminaries}\label{prelim}

\subsection{Generalized quadrangles.}

Let $s,t$ be positive integers. A {\em generalized quadrangle}\index{generalized quadrangle} of {\em order} $(s,t)$\index{generalized quadrangle!order} is an incidence structure $Q = (\mP,\mL,I)$ in which $\mP$ and $\mL$ are disjoint non-empty sets of objects called {\em points} and {\em lines} respectively, and for which $I$ is a symmetric point-line incidence relation satisfying the following axioms. First, each point is incident with $t+1$ lines and two distinct points are incident with at most one common line. Second, each line is incident with $s+1$ points  and two distinct lines are incident with at most one common point. Third, if
$(x,\ell) \in \mP \times \mL$ and $x \not \in \ell$, then there is a unique $(x',\ell') \in \mP \times \mL$ such that $x I \ell' I x' I \ell$. From the axioms,
a quadrangle of order $(s,t)$ contains $(s+1)(st+1)$ points and can only exist if $\sqrt{s} \leq t \leq s^2$ by a result of D. Higman \cite{Higman1,Higman2}, with a simple combinatorial proof given later by Cameron \cite{Cameron}. Recall that for a set $R$ of points, $R^{\perp}_{\circ}$ denotes the set of points that are not in $R$ and are collinear with all points in $R$. In the case of a thick and locally sparse quadrangle $Q$ of order $(s,t)$, the following statements can be derived from the axioms (see Payne and Thas~\cite{PT}):
if $R$ is a set of pairwise non-collinear points of $Q$, then
\begin{eqnarray}
\mbox{If }|R| = 1: &\mbox{ }& |R^{\perp}_{\circ}| = s(t + 1) \label{first}\\
\mbox{If }|R| = 2: &\mbox{ }& |R^{\perp}_{\circ}| = t + 1 \label{second} \\
\mbox{If }|R| = 3: &\mbox{ }& |R^{\perp}_{\circ}| \leq s + 1 \label{third}
\end{eqnarray}
The third statement is directly from the definition of a locally sparse quadrangle. These properties will be used extensively to prove Theorem \ref{main}.

\subsection{Elementary inequalities}

We collect here some elementary inequalities that will be used throughout the proof. From $-x - x^2 \leq \log (1 - x) \leq -x$ for $x \in [0,1/2]$ we obtain
for reals $x_i \in [0,1]$ summing to $N$ and with sum of squares equal to $M$:
\begin{equation}\label{ebound1}
e^{-N - M} \leq \prod_{i = 1}^n (1 - x_i) \leq e^{-N}.
\end{equation}
Furthermore in the special cases $x_1 = x_2 = \dots = x_n = x$ and $x_i = i/x$, we obtain the following pair of
asymptotic formulas:
\begin{eqnarray}
(1 - x)^n \sim e^{-nx} &\mbox{ }& \mbox{ if }nx^2 \rightarrow 0 \mbox{ as }n \rightarrow \infty
\label{ebound2} \\
f_n(x) \sim x^n &\mbox{ }& \mbox{ if }n^2/x \rightarrow 0 \mbox{ as }n \rightarrow \infty.
\label{fbound}
\end{eqnarray}

\subsection{Concentration inequalities.} We use three concentration inequalities in the proof of Theorem \ref{main}. The first is an easy consequence of Markov's Inequality for concentration
in the upper tail of $f_k(X)$ when $X$ is a non-negative integer-valued random variable.

\begin{proposition}\label{moments}
Let $k \in \mathbb N, \lambda \in \mathbb R_+$ with $\lambda \geq k$, and $X : \Omega \rightarrow \mathbb N$ be a random variable. Then
\[ \mathbb P(X \geq \lambda) \leq \frac{\mathbb E(f_k(X))}{f_k(\lambda)}.\]
\end{proposition}

\begin{proof}
Since $X$ is non-negative integer valued, the event $X \geq \lambda$ is contained in the event $f_k(X) \geq f_k(\lambda)$. Therefore
\[ \mathbb P(X \geq \lambda) \leq \mathbb P(f_k(X) \geq f_k(\lambda)) \leq \frac{\mathbb E(f_k(X))}{f_k(\lambda)}\]
by Markov's Inequality.
\end{proof}

A sum of independent random variables is concentrated according to the so-called Chernoff Bound~\cite{C}. We shall use the
Chernoff Bound in the following form. We write $X \sim \mbox{Bin}(n,p)$ to denote a binomial random variable with probability $p$
over $n$ trials.

\begin{proposition}\label{chernoff}
Let $X \sim \mbox{Bin}(n,p)$. Then for $\delta \in [0,1]$,
\[ \mathbb P(|X - pn| \geq \delta pn) \leq 2e^{-\delta^2 pn/2}.\]
\end{proposition}

The final concentration inequality is an extension of the Hoeffding-Azuma inequality~\cite{AlonSpencer} on martingales with
bounded differences:

\begin{proposition}\label{chalker}
Let $(Z_i)_{i = 0}^m$ be a martingale such that $\triangle_i := Z_{i + 1} - Z_i \geq -b$ and let $c = (c_1,c_2,\dots,c_m) \in \mathbb R_+^m$ and $\|c\|^2 = \sum_{i = 1}^m c_i^2$. Then for $\lambda \in \mathbb R_+$,
\[ \mathbb P(Z_m \leq Z_0 - \lambda) \leq e^{-\lambda^2/8\|c\|^2} + (1 + 2b) \sum_{i = 1}^{m-1} \mathbb P(\triangle_i < -c_i).\]
\end{proposition}

A simpler but less general form of this inequality was given by Shamir and Spencer~\cite{SS} and an
essentially equivalent version was proved using stopping times by Kim~\cite{K}. A simple proof is given in Chalker, Godbole, Hitczenko, Radcliff and Ruehr (see Lemma 1,~\cite{Chalk}).

\section{Proof of Theorem \ref{main}}\label{proof}

\subsection{A randomized algorithm}

We start by describing the two-round randomized algorithm that will produce a maximal partial ovoid in a locally sparse quadrangle.
Let $Q$ be a locally sparse quadrangle of order $(s,t)$ where $t \geq s \geq 2$ and $t \geq s(\log s)^{2\alpha}$. Let $\mathcal{P}$ and $\mathcal{L}$ denote the point set and line set of $Q$.
All limits and asymptotic notation in what follows is taken with respect to $s \rightarrow \infty$.

\medskip

{\bf First Round.} Fix a point $x \in \mathcal{P}$, and for each line $\ell$ through $x$, independently flip a coin with heads probability
\begin{equation}\label{p}
ps = \frac{s\log t - \alpha s \log \log s}{t}
\end{equation}
where $\alpha > 4$. Note $ps \in [0,1]$ since $t \geq s(\log s)^{2\alpha} \geq s\log t$ if $s$ is large enough.
On each line $\ell$ where the coin turned up heads, select uniformly a point of $\ell \backslash \{x\}$.
Let $S$ be the set of selected points, and let
\[ U = \mathcal{P} \; \backslash \; (S \cup \{x\})^{\bowtie}.\]
the points in $S^{\bowtie}$ are called {\em covered by $S$}, and the points not in $S^{\bowtie}$ are {\em uncovered}. Then $U$ consists of the
uncovered points which are not collinear with $x$. Note also that for any point $y \in x^{\perp}_{\circ}$,
\[ \mathbb P(y \in S) = p.\]
The points not in $S^{\bowtie}$ will be covered by a set $T$ of points so that
$S \cup T$ is a maximal partial ovoid: the set $T$ is described as follows.

\medskip

{\bf Second Round.} Let $x^* \in x^{\perp} \backslash S^{\bowtie}$. On each line $\ell \in \mathcal{L}$ through $x^*$ with $\ell \cap U \neq \emptyset$,
uniformly and randomly select a point of $\ell \cap U$. These points together with a point on the line through $x^*$ and $x$
distinct from $x^*$ and $x$ form a partial ovoid $T$, and $S \cup T$ is also clearly a partial ovoid. For Theorem \ref{main},
we must show that if $s$ is large enough, then there exists a choice of $S$ and $T$ such that $S \cup T$ is a maximal partial ovoid, and also $|S \cup T| \leq s(\log s)^{\alpha}$.

\subsection{Random variables}

We first show that in selecting $S$, Properties I -- III described below occur simultaneously a.a.s. as $s \rightarrow \infty$:
{\em \begin{center}
\begin{tabular}{ll}
{\sc Property I}. & For all lines $\ell \in \mathcal{L}$ disjoint from $x$, $|\ell \cap U| < \lceil \log s \rceil$.\\
{\sc Property II}. & For all $u \in x^{\perp} \backslash S$,  $|u^{\perp} \cap U| \lesssim s(\log s)^{\alpha}$ \\
{\sc Property III}. &  For $v,w \not \in S \cup \{x\}$ with $v,w$ non-collinear, $|\{v,w\}^{\perp}_{\circ} \cap U| \gtrsim (\log s)^{\alpha}$.
\end{tabular}
\end{center}
}
For $u \in x^{\perp}_{\circ}$, let $U(u)$ denote the set of points in $\mathcal{P} \backslash x^{\perp}$ which are not
covered by $S \backslash \{u\}$, and define the random variable:
\[ X_u = |u^{\perp} \cap U(u)|.\]
In the case $u \in x^{\perp} \backslash S$, note that $U(u) = U$, so then $X_u = |u^{\perp} \cap U|$.
For $v,w \in \mathcal{P} \backslash \{x\}$ non-collinear, let
$U(v,w)$ denote the set of points in $\mathcal{P} \backslash x^{\perp}$ which are not
covered by $S \backslash \{v,w\}$, and define the random variable:
\[ X_{vw} = |\{v,w\}^{\perp}_{\circ} \cap U(v,w)|.\]
In the case $v,w \not \in S \cup \{x\}$, $U(v,w) = U$ and so $X_{vw} = |\{v,w\}^{\perp}_{\circ} \cap U|$.
Therefore, to prove Properties II and III, it is sufficient to show that a.a.s.,
$X_u \lesssim s(\log s)^{\alpha}$ for all $u \in x^{\perp}_{\circ}$ and $X_{vw} \gtrsim (\log s)^{\alpha}$ for all pairs of non-collinear points
$v,w \in \mathcal{P} \backslash \{x\}$. Assuming this is done, Theorem \ref{main} is derived as follows.

\subsection{Proof of Theorem \ref{main} from Properties I -- III}

First we show $|S| \lesssim s\log t$ using the Chernoff Bound, Proposition \ref{chernoff}.
There are $t + 1$ lines through $x$, and we independently selected each line with probability $ps$ and then one point on each selected line.
So $|S| \sim \mbox{Bin}(t + 1,ps)$ and $\mathbb E(|S|) = ps(t + 1) \sim s\log t$.
By Proposition \ref{chernoff}, for any $\delta > 0$,
\[ \mathbb P(|S| \geq (1 + \delta)s\log t) \leq 2\exp(-\tfrac{1}{2}\delta^2 s\log t) \rightarrow 0.\]
Therefore a.a.s. $|S| \lesssim s\log t$. Assuming that a.a.s., $S$ satisfies Properties I -- III, we fix an instance of such a partial ovoid $S$ with $|S| \lesssim s\log t$.

\medskip

We now randomly select an additional set $T$ of points as follows. Let $x^* \in x^{\perp} \backslash S^{\bowtie}$. For each line $\ell$ through $x^*$ such that $\ell \cap U \neq \emptyset$, randomly and uniformly choose a point of $\ell \cap U$.
Let $T$ be the set of those chosen points together with one point $x^{+}$ distinct from $x$ and $x^*$ on the line through $x$ and $x^*$. Note that the line through $x$ and $x^*$ has at least one other point, since $s \geq 2$. Then no point in $T$ is collinear with any point in $S$,
and no two points in $T$ are collinear -- in particular, $x^+$ is not collinear with any point on a line through $x^*$ since $Q$ is a quadrangle.
We estimate $|S \cup T|$ using Property II. By Property II, $|T| \leq X_{x^*} + 1 \lesssim s(\log s)^{\alpha}$. Therefore
\[ |S \cup T| \leq |S| + X_{x^*} + 1 \lesssim s\log t + s(\log s)^{\alpha} \lesssim s(\log s)^{\alpha}\]
as required for Theorem \ref{main}. Finally, we show that a.a.s., $S \cup T$ is a maximal partial ovoid, using Properties I and III.

\medskip

Clearly all points on the line through $x$ and $x^*$ are covered by $x^+$, and by construction of $T$ all other points collinear with $x^{*}$ are covered. For $v \in (x^{\perp} \backslash S^{\bowtie}) \cup U$ not collinear with $x^*$, a.a.s., $X_{vx^*} \geq \tfrac{1}{2}(\log s)^{\alpha}$ by Property III when $s$ is large enough (here $x^*$ plays for instance the role of $w$ in the statement of Property III).
By Property I, for large enough $s$, the probability that $v$ is not collinear with any point in $T$ is at most
\[ \Bigl(\frac{\log s - 1}{\log s}\Bigr)^{X_{vx^*}} \leq \Bigl(1 - \frac{1}{\log s}\Bigr)^{\tfrac{1}{2}(\log s)^{\alpha}} \leq e^{-\tfrac{1}{2}(\log s)^3} < \frac{1}{s^5}\]
since $\alpha > 4$. Since $|\mP| = (s + 1)(st + 1) \lesssim s^4$ since $t \leq s^2$, the expected number of points in $(x^{\perp} \backslash S^{\bowtie}) \cup U$ not collinear with any point in $T$ is at most
\[ s^{-5}|\mP| \lesssim \frac{1}{s}.\]
It follows that a.a.s.,
\[ (x^{\perp} \backslash S^{\bowtie}) \cup U \subset T^{\bowtie}\]
which means that $T$ covers all the points that were not covered by $S$.
This shows $S \cup T$ is a maximal partial ovoid, and proves Theorem \ref{main}. \qed

\subsection{Expected values}

The first step in proving Properties I -- III is to compute $\mathbb E(X_u)$ and $\mathbb E(X_{vw})$.
We show now that $\mathbb E(X_u) \sim s(\log s)^{\alpha}$ and $\mathbb E(X_{vw}) \sim (\log s)^{\alpha}$:

\begin{lemma}\label{expect}
Let $u \in x^{\perp}_{\circ}$, and let $v,w \in \mathcal{P} \backslash \{x\}$ be a pair of non-collinear points. Then
\[ \mathbb E(X_u) \sim s(\log s)^{\alpha} \quad \mbox{ and } \quad \mathbb E(X_{vw}) \sim (\log s)^{\alpha}.\]
In addition, if $j \in \mathbb N$ and $jtp^2 \rightarrow 0$ as $s \rightarrow \infty$, then $\mathbb E(X_u)^j \sim s^j (\log s)^{\alpha j}$.
\end{lemma}

\begin{proof}
Fix $u \in x^{\perp}_{\circ}$ and recall $U(u)$ is the set of points in $\mathcal{P} \backslash x^{\perp}$
which are not covered by $S \backslash \{u\}$. For $y \in u^{\perp}_{\circ} \backslash x^{\perp}$, we have $|\{y,x\}^{\perp}_{\circ} \backslash \{u\}| = t$ by (\ref{second}).
Therefore
\begin{equation}\label{survive}
\mathbb P(y \in U(u)) = (1 - p)^{t}.
\end{equation}
In addition, $s$ points other than $u$ are collinear with both $u$ and $x$, namely the points on the unique
line through $u$ and $x$. By (\ref{first}), $|u^{\perp}_{\circ} \backslash x^{\perp}| = s(t + 1) - s = st$. Therefore for $u \in x^{\perp}_{\circ}$, by (\ref{survive}),
\begin{eqnarray*}
\mathbb E(X_u) &=& \sum_{y \in u^{\perp}_{\circ} \backslash x^{\perp}} \mathbb P(y \in U(u)) = st \cdot (1 - p)^{t}.
\end{eqnarray*}
Since $sp \rightarrow 0$ and $tp^2  \rightarrow 0$ by definition of $p$, we have from (\ref{ebound2}) that
\[ \mathbb E(X_u) \sim st \cdot \exp(-pt) \sim s(\log s)^{\alpha}.\]
Furthermore, by the same argument, and since $jtp^2 \rightarrow 0$, we have from (\ref{ebound2}) that
\[ \mathbb E(X_u)^j \sim s^j (\log s)^{\alpha j}.\]
To estimate $\mathbb E(X_{vw})$, for $v,w \in \mathcal{P} \backslash \{x\}$ not collinear, let $Y = \{v,w\}^{\perp}_{\circ} \backslash x^{\perp}$. Recall $U(v,w)$ is the set of points in $\mathcal{P} \backslash x^{\perp}$
which are not covered by $S \backslash \{v,w\}$. By (\ref{third}), $t - s \leq |Y| \leq t + 1$.
Observe for $y \in Y$, we have $|\{y,x\}^{\perp}_{\circ} \backslash \{v,w\}| = t + 1 - |\{v,w\} \cap x^{\perp}|$. Therefore for $y \in Y$,
\begin{equation}\label{survive2}
\mathbb P(y \in U(v,w)) \sim (1 - p)^{t + 1 - |\{v,w\} \cap x^{\perp}|} \sim (1 - p)^{t + 1}.
\end{equation}
By (\ref{survive2}) and the definition of $X_{vw}$,
\begin{eqnarray*}
\mathbb E(X_{vw}) &=& \sum_{y \in Y} \mathbb P(y \in U(v,w)) \sim |Y| \cdot (1 - p)^t \sim t \cdot \Bigl(\frac{(\log s)^{\alpha}}{t}\Bigr) = (\log s)^{\alpha}.
\end{eqnarray*}
This completes the proof that $\mathbb E(X_{vw}) \sim (\log s)^{\alpha}$.
\end{proof}

\subsection{Proof of Property I}

Fix a line $\ell \in \mathcal{L}$ disjoint from $x$, and let $Y_{\ell}$ be the number of sequences of $a = \lceil \log s\rceil$ distinct points
in $U \cap (\ell \backslash x^{\perp})$. Then since $Q$ is a quadrangle, every point not on $\ell$ is collinear with at most one point on $\ell$.
Let $R \subset \ell \backslash x^{\perp}$ be a set of $a$ distinct points. By (\ref{second}), $|\{x,y\}^{\perp}_{\circ}| = t + 1$ for all $y \in \ell \backslash x^{\perp}$. Therefore
\begin{equation}\label{manycollinear}
 \Bigl|\bigcup_{y \in R} \{x,y\}^{\perp}_{\circ}\Bigr| = at + 1
 \end{equation}
since there is a unique point in $\ell \cap x^{\perp}_{\circ}$ collinear with all points in $R$.
By (\ref{manycollinear}),
\[ \mathbb E(Y_{\ell}) = s(s - 1)(s - 2) \dots (s - a + 1) \cdot (1 - p)^{at + 1}.\]
Since $atp^2 \rightarrow 0$ and $a^2/s \rightarrow 0$, we may apply (\ref{ebound2}) and (\ref{fbound}) to obtain
\[ \mathbb E(Y_{\ell}) \sim \frac{s^a(\log s)^{a\alpha}}{t^a}.\]
Let $A_{s} = \displaystyle{\bigcup_{{\ell \in \mathcal{L}}\atop{x \not \in \ell}} [Y_{\ell} \geq 1]}$.
Since $|\mL| = (t + 1)(st + 1) \sim s t^2$ is the total number of lines,
\[ \mathbb P(A_{s}) \leq \sum_{{\ell \in \mathcal{L}}\atop{x \not \in \ell}} \mathbb P(Y_{\ell} \geq 1) \lesssim s t^2 \cdot \mathbb E(Y_{\ell}) \sim \frac{s^{a + 1}(\log s)^{a\alpha}}{t^{a - 2}}.\]
Since $t \geq s(\log s)^{2\alpha}$ and $a = \lceil \log s\rceil$, $\mathbb P(A_{s}) \rightarrow 0$ as $s \rightarrow \infty$, as required for Property I. \qed

\subsection{Proof of Property II}

By Lemma \ref{expect}, $\mu = \mathbb E(X_u) \sim s(\log s)^{\alpha}$ for all $u \in x^{\perp}_{\circ}$.
Let $j = \lceil (\log s)^2 \rceil$. To show $X_u \lesssim \mu$ a.a.s.,
we will apply Proposition \ref{moments} together with the following claim:

\begin{center}
\parbox{5.5in}{{\bf Claim.} {\it For any $u \in x^{\perp}_{\circ}$, $\mathbb E(f_{j}(X_u)) \lesssim \mu^{j}$.}}
\end{center}

{\it Proof of claim.} Fix a sequence $R$ of $j$ distinct points in $u^{\perp} \backslash x^{\perp}$, that is, collinear with $u$
but not on the line through $u$ and $x$, and let $R(u) = R^{\bowtie} \backslash \{u\}$. Let
\[ d = \sum_{\ell \ni x : u \not \in \ell} |R(u) \cap \ell|.\]
Then for a given line $\ell$ through $x$ but not through $u$,
\[ \mathbb P(S \cap R(u) \cap \ell = \emptyset) = 1 - p \cdot |R(u) \cap \ell|.\]
Since lines through $x$ are selected independently,
\[ \mathbb P(R \subset U(u)) = \prod_{\ell \ni x : u \not \in \ell} (1 - p|R(u) \cap \ell|).\]
Now we apply inequality (\ref{ebound1}) to obtain:
\[ \mathbb P(R \subset U(u)) \leq e^{-pd}.\]
By (\ref{second}), $|\{x,y\}^{\perp} \backslash \{u\}| = t$ for $y \in R$, and by (\ref{third}), $|\{x,y,z\}^{\perp}_{\circ}| \leq s + 1$ for $y,z \in R$. It follows by inclusion-exclusion that
\[ d \geq jt - {j \choose 2}(s + 1).\]
This shows
\[
\mathbb P(R \subset U) \leq e^{-pd} \leq e^{-pjt + pj^2(s + 1)}.
\]
Since $|u^{\perp}_{\circ} \backslash x^{\perp}| = st$ by (\ref{first}), there are at most $f_{j}(st)$ choices for the sequence $R$.
By definition of $X_u$,
\begin{eqnarray*}
\mathbb E(f_{j}(X_u)) \leq \sum_{{R \subset u^{\perp} \backslash x^{\perp}}\atop{|R| = j}} \mathbb P(R \subset U) \leq f_j(st) \cdot e^{-pjt + pj^2(s + 1)}.
\end{eqnarray*}
By definition of $p$ and since $t \geq s(\log s)^{\alpha}$, $pj^2(s + 1) \rightarrow 0$. Since $pjt = j\log t - \alpha j \log \log s$,
\[ \mathbb E(f_{j}(X_u)) \lesssim s^j (\log s)^{\alpha j}.\]
By Lemma \ref{expect}, since $jtp^2 \rightarrow 0$, we have $\mu^j \sim s^j(\log s)^{\alpha j}$, as required for the claim. $\diamond$

\medskip

To prove Property II, we combine the claim with Proposition \ref{moments}, where $\lambda = \mu/(1 - \delta)$ and $\delta \in (0,1)$ is a fixed
positive constant independent of $s$. By Proposition \ref{moments},
\[ \mathbb P(X_u \geq \lambda) \leq \frac{\mathbb E(f_j(X_u))}{f_j(\lambda)} \lesssim (1 - \delta)^{j} \leq e^{-\delta j}\]
by (\ref{ebound1}). Since $j = \lceil (\log s)^2 \rceil$, the above quantity is at most $s^{-\delta \log s}$ provided $s$ is
large enough. So the expected number of $u$ such that the event $[X_u \geq \lambda]$ occurs is at most asymptotic to
\[ s^{-\delta \log s} \cdot |u^{\perp}_{\circ}| \sim s^{-\delta \log s} \cdot st \rightarrow 0\]
since $t \leq s^2$. By Markov's Inequality,
\[ \mathbb P\Bigl(\bigcup_{u \in x_{\circ}^{\perp}} [X_u \geq \tfrac{\mu}{1 - \delta}]\Bigr) \rightarrow 0\]
as $s \rightarrow \infty$. Since $\delta$ was arbitrary,  
we find that a.a.s, $X_u \lesssim s(\log s)^{\alpha}$ for all $u \in x_{0}^{\perp}$. \qed

\subsection{Proof of Property III}

By Lemma \ref{expect}, $\nu = \mathbb E(X_{vw}) \sim (\log s)^{\alpha}$ for all pairs of non-collinear points $v, w \in \mathcal{P} \backslash \{x\}$. We show that a.a.s., $X_{vw} \gtrsim \nu$, using Proposition \ref{chalker} on a carefully chosen martingale. Fix a pair of non-collinear points $v, w \in \mathcal{P} \backslash \{x\}$, and recall
$U(v,w)$ is the set of points in $\mathcal{P} \backslash x^{\perp}$ which are not covered by $S \backslash \{v,w\}$. Let $\mathfrak{L}$
be the set of lines through $x$ which contain neither $v$ nor $w$, so $t - 1 \leq |\mathfrak{L}| \leq t + 1$ by (\ref{second}).

\medskip

{\bf Notation.} Let $\beta$ satisfy $\alpha - 1 > \beta > 3$, $r = |\mathfrak{L}|/(\log s)^{\beta}$, and let $m = (\log s)^{\beta}$ -- we assume $(\log s)^{\beta}$ is an integer -- this does not affect the asymptotic computations below.
Let $\mathfrak{L}_1,\mathfrak{L}_2,\dots,\mathfrak{L}_{m}$ be a partition of $\mathfrak{L}$ into sets of $r$ lines, and $\mathfrak{M}_i := \mathfrak{L}_1 \cup \mathfrak{L}_2 \cup  \dots \cup \mathfrak{L}_i$. Let $L_i = \bigcup_{\ell \in \mathfrak{L}_i} \ell$
and $M_i = \bigcup_{\ell \in \mathfrak{M}_i} \ell$ for $i \leq m$. For $z \in x^{\perp}_{\circ}$, let $\chi_z = 1$ if $z \in S$ and $\chi_z = 0$ otherwise.

\medskip

{\bf Definition of a martingale.} Let $Z = X_{vw}$ and let $Z_i = \mathbb E(Z|\mathcal{F}_i)$ where for $0 \leq i \leq m$, $\mathcal{F}_i$ is the $\sigma$-field generated by the random variables
$\{\chi_z : z \in M_i\}$ with $M_0 = \emptyset$. Then $(Z_i)_{i = 0}^{m}$ is a Doob martingale with difference sequence $\triangle_i = Z_{i + 1} - Z_i$ for $0 \leq i < m$
and $Z_0 = \mathbb E(Z) = \nu$. Let $I = \{v,w\}^{\perp}_{\circ} \backslash x^{\perp}$. For $y \in I$, let $\phi_y = 1$ if $y \in U(v,w)$ and $\phi_y = 0$ otherwise, and define
\[ \triangle_i(y) = \mathbb{E}(\phi_y|\mathcal{F}_{i+1}) - \mathbb{E}(\phi_y|\mathcal{F}_i).\]
Then by definition of $Z$,
\[ \triangle_i = \sum_{y \in I} \triangle_i(y).\]

{\bf Distribution of differences.}
For $y \in I$, we note $|\{x,y\}^{\perp}_{\circ} \cap L_i| = |\mathfrak{L}_i| = r$ for all $i \leq m$. Therefore for $y \in I$,
\begin{eqnarray*}
\triangle_i(y) &=& \prod_{z \in y^{\perp} \cap M_{i + 1}} (1 - \chi_z) (1 - p)^{|\mathfrak{L}| - |\mathfrak{M}_{i + 1}|} - \prod_{z \in y^{\perp} \cap M_i} (1 - \chi_z) (1 - p)^{|\mathfrak{L}| - |\mathfrak{M}_i|} \\
&=& \prod_{z \in y^{\perp} \cap M_i} (1 - \chi_z)(1 - p)^{|\mathfrak{L}| - |\mathfrak{M}_i|} \Bigl(\frac{\prod_{z \in y^{\perp} \cap L_{i + 1}} (1 - \chi_z)}{(1 - p)^r} - 1\Bigr).
\end{eqnarray*}
Since $|\mathfrak{M}_i| = ir$, we may use the above explicit formula for $\triangle_i(y)$ to determine the distribution of $\triangle_i(y)$, according to the following table:
the first column is the value of $\triangle_i(y)$ and the second column is the probability that $\triangle_i(y)$ equals the value in the first column for $y \in I$:
\[ \triangle_i(y) = \left\{\begin{array}{ll}
-(1 - p)^{|\mathfrak{L}| - ir} & (1 - p)^{ir} - (1 - p)^{(i + 1)r} \\
((1 - p)^{-r} - 1) (1 - p)^{|\mathfrak{L}| - ir} \mbox{\hspace{0.5in}}  & (1 - p)^{(i + 1)r} \\
0 & 1 - (1 - p)^{ir}
\end{array}\right.\]
We note that $|\mathfrak{L}| \in \{t-1,t,t+1\}$.
The random variable $\triangle_i(y)$ is not as easy to deal with as $\tilde{\triangle}_i(y) = \min\{\triangle_i(y),0\}$ for $y \in I$; the distribution of the latter is
\[ \tilde{\triangle}_i(y) = \left\{\begin{array}{ll}
-(1 - p)^{|\mathfrak{L}| - ir} \mbox{\hspace{0.5in}} & (1 - p)^{ir} - (1 - p)^{(i + 1)r} \\
0   & 1 - (1 - p)^{ir} + (1 - p)^{(i + 1)r} \\
\end{array}\right.\]
Let $\tilde{\triangle}_i = \sum_{y \in I}  \tilde{\triangle}_i(y)$. Note that $\mathbb E(\tilde{\triangle}_i(y)) < 0$.

\medskip

{\bf Expected values.} By definition of $\tilde{\triangle}_i(y)$ and the choice of $r$,
\begin{eqnarray*}
\mathbb E(\tilde{\triangle}_i(y)) &=& -(1 - p)^{|\mathfrak{L}| - ir} (1 - p)^{ir} (1 - (1 - p)^r) \\
&=& -(1 - p)^{|\mathfrak{L}|} (1 - (1 - p)^r) \\
&\sim& -(1 - p)^{|\mathfrak{L}|} pr.
\end{eqnarray*}
By (\ref{ebound2}) and $|\mathfrak{L}| \in \{t-1,t,t+1\}$, $(1 - p)^{|\mathfrak{L}|} \sim (\log s)^{\alpha}/t$, and so
\begin{eqnarray*}
-(1 - p)^{|\mathfrak{L}|} pr &\sim& -\frac{(\log s)^{\alpha}}{t} \cdot \frac{\log t}{(\log s)^{\beta}} \\
 &=& -\frac{(\log s)^{\alpha - \beta}(\log t)}{t}.
\end{eqnarray*}
Then since $|I| = |\{v,w\}^{\perp}_{\circ} \backslash x^{\perp}| \sim t$ by (\ref{second}) and (\ref{third}),
\begin{eqnarray*}
\mu_i := \mathbb E(\tilde{\triangle}_i) = \sum_{y \in I} \mathbb E(\tilde{\triangle}_i(y)) \sim -(\log s)^{\alpha - \beta}(\log t).
\end{eqnarray*}

\medskip

{\bf Concentration.} This is the main part of the proof of Property III. We use Proposition \ref{moments} to show that $\tilde{\triangle}_i$ is highly unlikely to drop substantially below its expectation $\mu_i$.
Note throughout that $\mu_i < 0$. Let $k = \lceil(\log s)^{\alpha - \beta}(\log t)\rceil$ and $\varepsilon = (\log s)^{\beta - \alpha}(\log t)$.
Note $\varepsilon \rightarrow 0$ as $s \rightarrow \infty$, since $\alpha > \beta + 1$ and $t \leq s^2$. We prove the following claim:

\begin{center}
\parbox{5.5in}{{\bf Claim.} {\it For all $i < m$, $\mathbb P(\tilde{\triangle}_i < (1 + \varepsilon)\mu_i) \lesssim  t^{-(\log t)/2}$.}}
\end{center}

{\it Proof of claim.} Let $Y$ be the number of $y \in I$ such that $\chi_z = 1$ for some $z \in y^{\perp} \cap L_{i + 1}$ and
$\chi_z = 0$ for all $z \in y^{\perp} \cap M_i$. Note that since $|I| \sim t$ by (\ref{second}) and (\ref{third}),
\begin{eqnarray}
\mathbb E(Y) &=& \sum_{y \in I} (1 - p)^{ir} (1 - (1 - p)^r) \nonumber \\
&\sim& t(1 - p)^{ir}(1 - (1 - p)^r) \; \; \sim \; \; t(1 - p)^{ir} \cdot pr. \label{eofy}
\end{eqnarray}
Also note that $\tilde{\triangle}_i = -Y(1 - p)^{|\mathfrak{L}| - ir}$ by definition of $\tilde{\triangle}_i(y)$ for $y \in I$.
It follows that
\[ [\tilde{\triangle}_i \leq (1 + \varepsilon)\mu_i] = [Y \geq (1 + \varepsilon)\mathbb E(Y)].\]
Let $\lambda = (1 + \varepsilon) \mathbb E(Y)$.
Let $S_{i + 1}$ be the set of selected points in $L_{i + 1}$ and $W = |S_{i + 1}|$.
Let $\sigma(R) = 1$ if $R \subseteq S_{i + 1}^{\bowtie} \cap I$ and
$\sigma(R) = 0$ otherwise, and let $\tau(R) = 1$ if for every $y \in R$, $y^{\perp} \cap M_i \cap S = \emptyset$,
and $\tau(R) = 0$ otherwise.  We apply Proposition \ref{moments}:
\[ \mathbb P(Y \geq \lambda) \leq \frac{\mathbb E(f_k(Y))}{f_k(\lambda)}.\]
By definition of $f_k(Y)$,
\[ f_k(Y) = \sum_{R \subset I} \sigma(R) \tau(R)\]
where the sum is over all sequences $R$ of $k$ points in $I$. Note that the contribution to the sum
of sequences $R$ which are not subsets of $S_{i + 1}^{\bowtie} \cap I$ is zero, by definition of $\sigma(R)$ and $\tau(R)$.
For any sequence $R$ of $k$ points in $I$, each $y \in R$ is collinear with exactly one point on each line in $\mathfrak{M}_{i}$, so in total
$y$ is collinear with $ir$ points in $M_i$, and $r$ points in $L_{i + 1}$. By inclusion-exclusion,
\[ \Bigl|\bigcup_{y \in R} y^{\perp} \cap M_i\Bigr| \geq k(ir) - {k \choose 2}(s + 1)\]
since any two points in $I$ are collinear with at most $s + 1$ points in $M_i$, by (\ref{third}).
Therefore
\[ \mathbb P(\tau(R) = 1) \leq (1 - p)^{k(ir) - {k \choose 2}(s + 1)} \leq (1 - p)^{kir}e^{pk^2(s + 1)}. \]
Since $pk^2(s + 1) \leq (\log t)k^2(s + 1)/t$ and $t \geq s(\log s)^{2\alpha}$, we find
\[ \mathbb P(\tau(R) = 1) \lesssim (1 - p)^{kir}.\]
Now since $\sigma(R)$ and $\tau(R)$ are independent random variables,
\[ \mathbb E(f_k(Y)) = \sum_{R \subset I} \mathbb P(\sigma(R) = 1) \mathbb P(\tau(R) = 1) \lesssim  (1 - p)^{kir} \cdot \sum_{R \subset I} \mathbb P(\sigma(R) = 1).\]
Now $\sum_{R \subset I} \sigma(R)$ counts the number of sequences of $k$ distinct points in $I$ each collinear with at least one point in $S_{i + 1}$. If $R$ comprises points $y_1,y_2,\dots,y_k \in I$, then there exists a smallest positive integer $\kappa \leq k$ such that $S_{i + 1}$
contains $\kappa$ distinct points $z_1,z_2,\dots,z_{\kappa}$ with $z_i \in y_i^{\perp}$ for all $i \leq \kappa$ and $y_{\kappa + 1},y_{\kappa + 2},\dots,y_k$ are each collinear with at least one of $z_1,z_2,\dots,z_{\kappa}$. There are at most $(t + 1)^{\kappa}$ choices for $y_1,y_2,\dots,y_{\kappa} \in I$. Since $|y_i^{\perp} \cap x^{\perp} \cap L_{i+1}| = r$ for all $i$, there are at most $r$ choices for $z_i$
for $1 \leq i \leq \kappa$. Finally, since each $y_i : i > \kappa$ is collinear with at least one of $z_1,z_2,\dots,z_{\kappa}$,
there are at most
\[ \kappa^{k - \kappa} |\{v,w,z_i\}^{\perp}|^{k - \kappa} \leq \kappa^{k - \kappa} (s + 1)^{k - \kappa}\]
choices for $y_{\kappa + 1},y_{\kappa + 2},\dots,y_k$ by (\ref{third}). Each such configuration
has probability at most $p^{\kappa}$. Therefore
\begin{eqnarray*}
\sum_{R \subset I} \mathbb P(\sigma(R) = 1) &\leq& \sum_{\kappa = 1}^{k} (pr)^{\kappa}(t + 1)^{\kappa} \kappa^{k - \kappa} (s + 1)^{k - \kappa} \\
&\leq& (pr)^k(t + 1)^k \sum_{\kappa = 0}^k \theta^{\kappa} \\
&=& (pr)^k(t + 1)^k \cdot \frac{1 - \theta^{k + 1}}{1 - \theta}
\end{eqnarray*}
where $\theta = k(s + 1)/pr(t + 1)$. Since $t \geq s(\log s)^{2\alpha}$ and $r = |\mathfrak{L}|/(\log s)^{\beta}$ and $k = \lceil (\log s)^{\alpha - \beta}\log t \rceil$, we deduce
\[ \theta \lesssim \frac{ks}{r\log t} \lesssim \frac{(\log s)^{\alpha}s}{t} \rightarrow 0\]
as $s \rightarrow \infty$ and therefore
\[ \sum_{R \subset I} \mathbb P(\sigma(R) = 1) \lesssim (prt)^k.\]
As $(t + 1)^k \sim t^k$ by the choice of $k$, it follows from (\ref{eofy}) that
\begin{eqnarray*}
\mathbb E(f_k(Y)) &\lesssim& (1 - p)^{kir} \cdot (prt)^k \; \; \lesssim \; \; \mathbb E(Y)^k.
\end{eqnarray*}
Since $f_{k}(\lambda) \sim (1 + \varepsilon)^k \mathbb E(Y)^{k}$ by (\ref{fbound}), Proposition \ref{moments} gives
\[ \mathbb P(Y \geq \lambda) \lesssim (1 + \varepsilon)^{-k}.\]
For large enough $s$, $\varepsilon \leq 1$, and so $\log(1 + \epsilon) \geq \varepsilon - \varepsilon^2/2$. 
Therefore for large enough $s$, 
\[ \log (1 + \varepsilon)^{-k} \leq -\varepsilon k + \frac{1}{2}\varepsilon^2 k \leq -\frac{1}{2}(\log t)^2.\]
It follows that
\[ \mathbb P(\tilde{\triangle}_i < (1 + \varepsilon)\mu_i) \leq \mathbb P(Y \geq \lambda) \lesssim (1 + \varepsilon)^{-k} \leq t^{-(\log t)/2}. \]
This proves the claim. $\diamond$

\bigskip

{\bf Proof of Property III.} Since $\triangle_i$ majorizes $\tilde{\triangle_i}$ and $\mu_i < 0$, we conclude from the claim that
\[ \mathbb P(\triangle_i < (1 + \varepsilon)\mu_i) \leq \mathbb P(\tilde{\triangle}_i < (1 + \varepsilon)\mu_i) \lesssim t^{-(\log t)/2}.\]
Now we apply Proposition \ref{chalker}. Let
\[ c_i = |(1 + \varepsilon) \mu_i| \sim (1 + \varepsilon)(\log s)^{\alpha - \beta}(\log t)\]
for $0 \leq i < m$. Then
\[ \sum_{i = 0}^{m - 1} c_i^2 \sim (\log s)^{\beta} (1 + \varepsilon)^2 (\log s)^{2\alpha - 2\beta}(\log t)^2 \sim (\log s)^{2\alpha - \beta}(\log t)^2.\]
Let $b = t+1$. Then $\triangle_i \geq -b$ for $0 \leq i < m$. So for any $\zeta \in (0,1)$ independent of $s$,
\begin{eqnarray*}
\mathbb P(X_{vw} < (1 - \zeta)\nu) &\leq& \exp\Bigl(-\frac{\zeta^2 \nu^2}{8\|c\|^2}\Bigr) + (1 + 2b) \sum_{i = 0}^{m - 1} \mathbb P(\triangle_i < -c_i).
\end{eqnarray*}
Since $\nu \sim (\log s)^{\alpha}$, the exponent in the first term is asymptotic to
\[ -\frac{\zeta^2}{8}(\log s)^{2\alpha}(\log s)^{-2\alpha + \beta}(\log t)^{-2} \lesssim -10\log s\]
since $\beta > 3$ and $t \leq s^2$. Since $m = (\log s)^{\beta}$, the second term is
\[ (1 + 2b) \sum_{i = 0}^{m-1} \mathbb P(\triangle_i < -c_i) \lesssim (1 + 2b) \cdot mt^{-(\log t)/2} \lesssim s^{-10}\]
using the claim.  So for any $\zeta \in (0,1)$, and a fixed pair $v,w \in \mathcal{P} \backslash \{x\}$ of non-collinear points,
\[ \mathbb P(X_{vw} < (1 - \zeta)\nu) \lesssim 2s^{-10}.\]
Therefore the expected number of pairs $v,w \in \mathcal{P} \backslash \{x\}$ of non-collinear points
such that $X_{vw} < (1 - \zeta)\nu$ is at most $|\mP|^2 s^{-10} \lesssim s^4 t^2 s^{-10} \leq s^{-2}$, since $t \leq s^2$. By Markov's Inequality,
a.a.s., every pair $v,w \in \mathcal{P} \backslash \{x\}$ of non-collinear points has
$X_{vw} \geq (1 - \zeta)\nu$. Since this is valid for arbitrary $\zeta > 0$ and $\nu \sim (\log s)^{\alpha}$,
Property III is proved. \qed

\section{Concluding remarks}

$\bullet$ In this paper we showed that any locally sparse quadrangle of order $(s,t)$ with $t \geq s(\log s)^{2\alpha}$ and $s$ large enough has a maximal partial ovoid of size at most $s(\log s)^{\alpha}$ when $\alpha > 4$. The main obstruction to carrying out the same proof as Theorem \ref{main} for general quadrangles with $t$ large enough relative to $s$ is the locally sparse condition -- in particular in a quadrangle of order $(s,t)$ even with $t$ large relative to $s$, there may be sets of three points all collinear with up to $t + 1$ other points. For instance, in $H(4,q^2)$, which is a quadrangle of order $(s,t) = (q^2,q^3)$, this situation arises~\cite{PT}. Nevertheless we
pose the following problem for all quadrangles:

\medskip

{\bf Problem 1.} {\em Does every generalized quadrangle of order $(s,t)$ have a maximal partial ovoid of size at most $s \cdot \mbox{polylog}(s)$ as $s \rightarrow \infty$?}

\medskip

This should be compared with the easy linear lower bound $(s + 1)(st + 1)/s(t + 1) \sim s$, which has been slightly improved to a larger constant times $s$
in a number of special cases (see De Beule, Klein, Metsch and Storme~\cite{deBKMS} and the references therein).

$\bullet$ In some cases, superlinear lower bounds on the size of a maximal partial ovoid in a quadrangle of order $(s,t)$ may be provable
using explicit algebraic representations over finite fields. This typically involves using known character sum inequalities, for instance as in G\'acs and Sz\"{o}nyi~\cite{GS}. We leave the open problem of determining whether a superlinear lower bound can be achieved in $Q^-(5,q)$:

\medskip

{\bf Problem 2.} {\em Is there a function $\omega : \mathbb N \rightarrow \mathbb R_+$ such that $\omega(q) \rightarrow \infty$ as $q \rightarrow \infty$ and every maximal partial ovoid in $Q^-(5,q)$ has size at least $q\omega(q)$?}

\bigskip

$\bullet$ The randomized algorithm in this paper is very simple to implement, and we believe it is effective in finding maximal partial ovoids even in $(s,t)$-quadrangles where $s$ is not too large. In addition, it can be deduced from the proof that the probability that the algorithm does not return a maximal partial ovoid of size at most $s(\log s)^\alpha, \alpha>4$, is at most $s^{-\log s}$ if $s$ is large enough.

\section{Acknowledgement}

We would like to thank J.A. Thas for pointing out several useful facts on generalized quadrangles.

\end{document}